\documentclass[10pt]{amsart}
\usepackage{amssymb}
\usepackage{times}
\usepackage{url}
\usepackage{hyperref}
\usepackage{xcolor}

\usepackage{rotating}
\usepackage{pdflscape}

\usepackage{mathrsfs}

\newtheorem{theorem}{Theorem}[section]
\newtheorem{lemma}[theorem]{Lemma}

\newtheorem{proposition}[theorem]{Proposition}
\newtheorem{corollary}[theorem]{Corollary}
\theoremstyle{remark}
\newtheorem{remark}[theorem]{Remark}

\newcommand{\F}{\mathbb{F}}
\newcommand{\Z}{\mathbb{Z}}
\newcommand{\Q}{\mathbb{Q}}

\newcommand{\C}{\mathbb{C}}

\newcommand{\GL}{\mathrm{GL}}

\newcommand{\SL}{\mathrm{SL}}

\newcommand{\Sp}{\mathrm{Sp}}

\newcommand{\Aut}{\mathrm{Aut}}

\newcommand{\trace}{\mathrm{trace}}
\newcommand{\rk}{\mathrm{rk}}
\newcommand{\ad}{\mathrm{ad}}

\newcommand{\calc}{\mathscr{C}}

\newcommand{\cals}{\mathscr{S}}
\newcommand{\calr}{\mathcal{A}}

\definecolor{grey}{rgb}{0.6,0.6,0.6}

\begin{document}

\title[The strong approximation theorem and computing 
with linear groups]{The strong approximation theorem and 
computing with linear groups}

\author{A.~S.~Detinko}
\address{School of Computer Science\\
University of St~Andrews\\
North Haugh\\
St~Andrews KY16 9SX\\
UK}
\email{ad271@st-andrews.ac.uk}

\author{D.~L.~Flannery}
\address{
School of Mathematics, Statistics and Applied Mathematics\\
National University of Ireland, Galway\\
University Road\\
Galway H91TK33\\
Ireland}
\email{dane.flannery@nuigalway.ie}

\author{A.~Hulpke}
\address{Department of Mathematics\\
Colorado State University\\
Fort Collins\\
CO 80523-1874\\
USA}
\email{Alexander.Hulpke@colostate.edu}


\begin{abstract}
We obtain a computational realization of the strong approximation theorem. 
That is, we develop algorithms to compute all congruence quotients modulo 
rational primes of a finitely generated Zariski  dense group 
$H \leq \mathrm{SL}(n, \mathbb{Z})$ for $n \geq 2$. More generally, we 
are able to compute all congruence quotients of a finitely generated Zariski 
dense subgroup of  $\mathrm{SL}(n, \mathbb{Q})$ for $n > 2$.
\end{abstract}

\maketitle

\section{Introduction}

The \emph{strong approximation theorem} (SAT) is a 
milestone of linear group theory and its 
applications~\cite[Window~9]{LubotzkySegal}.
It has come to play a similarly important role in 
computing with linear groups~\cite{DensityFurther}. 

Let $H$ be a finitely generated subgroup of 
$\SL(n, \Z)$ that is Zariski dense in 
$\SL(n,\C)$. Then SAT asserts that $H$ is congruent to 
$\SL(n, p)$ for all but a finite number of primes $p\in \Z$. 
Therefore, we can describe the congruence 
quotients of $H$ modulo all primes. Moreover,
we can describe the congruence quotients of $H$ modulo all 
positive integers if $n>2$ (see \cite[Section~4.1]{DensityFurther}). 

The congruence quotients of $H$ provide important 
information about $H$; especially when $H$ is arithmetic, 
i.e., of finite index in $\SL(n, \Z)$. In 
that case, the set $\Pi(H)$ of all primes $p$ such that 
$H \not\equiv \SL(n, p)$ modulo $p$ 
is (apart from some exceptions for $p = 2$ and $n \leq 4$) 
the set of primes dividing the \emph{level} of $H$, 
defined to be the level of the unique maximal principal 
congruence subgroup in $H$ \cite[Section~2]{Density}. 
If $H$ is \emph{thin}, i.e., dense but of infinite index in 
$\SL(n, \Z)$, then we consider the \emph{arithmetic closure} 
$\mathrm{cl}(H)$ of $H$: this is the intersection of all arithmetic 
groups in $\SL(n, \Z)$ containing $H$~\cite[Section~3]{Density}. 
Note that $\Pi(H) = \Pi(\mathrm{cl}(H))$ determines the level 
of $\mathrm{cl}(H)$ just as it does when $H$ is arithmetic.
The level is a key component of subsequent algorithms 
for computing with arithmetic subgroups, such as membership 
testing and orbit-stabilizer algorithms~\cite{Arithm}.

In \cite[Section~3.2]{Density} and \cite{DensityFurther}, we 
developed algorithms to compute $\Pi(H)$ when $n$ is prime 
or $H$ has a known transvection. This paper presents 
a complete solution: practical algorithms  
to compute $\Pi(H)$ for arbitrary finitely generated dense 
$H \leq \SL(n, \Z)$, $n \geq 2$. We also give a characterization 
of density that allows us to compute $\Pi(H)$ without preliminary 
testing of density (although this can certainly be done; see
\cite[Section~5]{Density} and \cite{OWPreprint17}). Our methods 
extend in a straightforward manner to handle input
$H\leq \SL(n, \Q)$.

As in \cite{DensityFurther}, we rely on the classification of 
maximal subgroups of $\SL(n, p)$. Specifically, we follow the proof
of SAT in \cite[Window~9, Theorem~10]{LubotzkySegal}, which
credits C.~R.~Matthews, L.~N.~Vaserstein,
and B.~Weisfeiler.
In Section~\ref{2} we prove results about maximal subgroups of 
$\SL(n, p)$ that are needed for the main algorithms. 
Then Section~\ref{3} provides methods to compute $\Pi(H)$ for 
dense $H \leq \SL(n, \Q)$. In Section~\ref{4} we outline the 
algorithms, and in Section~\ref{5} demonstrate their practicality. 

We now fix some basic terms and notation. Let 
$S = \{g_1, \ldots , g_r\}\subseteq \SL(n,\Q)$ 
and $H=\langle S\rangle$. Then
$R$ is the ring (localization) $\frac{1}{\mu}\Z$ generated by 
the entries of the $g_i$ and $g_i^{-1}$;
here $\mu$ is a positive integer. Note that $R$ 
depends only on $H$, not on the choice of generating set $S$ for $H$. 
For $m$ coprime to $\mu$, the congruence homomorphism $\varphi_m$ 
induced by natural surjection $\Z\rightarrow \Z_m = \Z/m\Z$
maps $\SL(n, R)$ onto $\SL(n, \Z_m)$. 
Let $\Pi(H)$ be the set of all primes $p$ (not dividing 
$\mu$) such that $\varphi_p(H) \neq \SL(n, p)$. 
Overlining will denote the image modulo a prime $p$ 
of an element of $R$ or a matrix or set of matrices over $R$.
In particular, $\bar{H} = \langle \bar{S} \rangle= \varphi_p(H)$. 
If $\bar{h}\in \bar{H}$ is given as a word 
$\Pi_i{\bar{g}_{j_i}^{e_i}}$ in $\bar{S}$,  then the `lift' of 
$\bar{h}$ is its preimage $h = \Pi_i{g_{j_i}^{e_i}}$.

Throughout, $\F$ is a field, $\F_p$ is the field of size $p$,
$\mathrm{Mat}(n,\F)$ is the $\F$-algebra of $n \times n$ matrices 
over $\F$, and $1_n \in \mathrm{Mat}(n,\F)$ is the identity matrix. 
We write $\langle G \rangle_D$ for the enveloping 
algebra of $G \leq \GL(n, \F)$ over a subring $D\subseteq \F$.

\section{Maximality of subgroups in $\SL(n, p)$}
\label{2}
Let $G\leq \SL(n,p)$. We show how
to recognize when $G$ is not in any maximal 
subgroup of $\SL(n, p)$, i.e., when $G = \SL(n, p)$. 
Our approach, which characterizes maximal subgroups by means of 
the adjoint representation, is motivated 
by \cite[Window~9, Section~2]{LubotzkySegal}.

We identify the adjoint module for $\SL(n,\F)$ with the $\F$-space 
\[
\mathfrak{sl}(n,\F) =\{x\in \mathrm{Mat}(n,\F)\mid\trace(x)=0\}
\] 
of dimension $n^2-1$ on which $\SL(n,\F)$ acts by conjugation. 
Let $\mathrm{ad}:\SL(n,\F)\rightarrow \GL(n^2-1,\F)$
be the corresponding linear representation.

The set of maximal subgroups of $\SL(n, p)$ is the union of 
Aschbacher classes $\calc_1, \ldots , \calc_8, \cals$ 
(see \cite{Asch84} and \cite[p.~397]{LubotzkySegal}).  
The classes $\calc_4$ and $\calc_7$ involve tensor products, 
for which we adopt the following convention.
If $H_1\le\GL(a,\F)$ and $H_2\le\GL(b,\F)$ then 
$H_1\times H_2$ acts on $\F^a\otimes\F^b$. The 
associated matrix representation of degree $a b$ has 
$(h_1,h_2)\in H_1 \times H_2$ acting as the matrix Kronecker 
product $h_1\dot{\times} h_2$. The group generated by these 
Kronecker products is denoted $H_1\otimes H_2$.
\begin{proposition}\label{absirrasch}
Let $G$ be a proper absolutely irreducible subgroup 
of $\SL(n, p)$ such that $\ad(G)$ is irreducible. 
Then $G$ lies in a maximal subgroup in ${\calc}_6\cup \cals$.
\end{proposition}
\begin{proof} 
Since $G$ is absolutely irreducible, it cannot be in a 
subgroup in ${\calc}_1$. Class ${\calc}_5$ is irrelevant 
over a field of prime size. For each of the 
remaining Aschbacher classes other than ${\calc}_6$ or $\cals$, 
we identify a proper submodule $T$ of the adjoint module 
$A$ for $\SL(n,p)$.
\begin{itemize}\allowbreak
\item[${\calc}_2$.] A maximal subgroup lies in $W=\GL(a,p)\wr S_b$ 
with $n=ab$. Let $T\le A$ be the subspace spanned by block matrices 
with $b$ blocks from $\{1_a,0_a,-1_a\}$ and zero trace. Clearly $T$ 
is preserved under conjugation by $W$ and has dimension $b-1$. 
\item[${\calc}_3$.]
A maximal subgroup here has a normal subgroup $N\cong\SL(a,p^b)$ 
with $n=ab$, $1<a,b<n$. Each `entry' of $N$ is a $b\times b$ 
submatrix. The set of matrices in the center of $N$ 
with trace $0$ is a proper submodule of $A$.
\item[${\calc}_4$.]
A maximal subgroup $L$ is $\SL(a,p)\otimes\SL(b,p)$ for some 
$a, b<n$ such that $n=ab$.
If $x\in\mathfrak{sl}(a,p)$ and $y=x\dot\times 1_b$
then $\trace(y)=0$ and thus $y\in A$. Let $T$ be the space spanned 
by all such products. Then $L$ acts on $T$ by the adjoint action 
of the $\SL(a,p)$-part of elements on the $x$-components of such 
products. Thus $T\le A$ is invariant under $L$, so is a proper 
submodule of $A$.
\item[${\calc}_7$.] We use an argument similar to the preceding 
one. Here a maximal subgroup is generated by $\mathrm{Sym}(b)$ 
and $\SL(a,p)\otimes\cdots\otimes\SL(a,p)$ with $b$ factors, 
where $n=a^b$ and $1<a,b<n$.
Let $T$ be the subspace of $A$ spanned by all 
Kronecker products of length $b$ with every factor $1_a$ except 
for one, drawn from the adjoint module of $\SL(a,p)$. Then $T$ is 
invariant under action by the maximal subgroup.
\item[${\calc}_8$.] 
A maximal subgroup that stabilizes a form preserves
its own adjoint module (see, e.g., \cite[p.~398]{LubotzkySegal} 
or \cite[Section~1.4.3]{Goodman}), which cannot be $A$. \qedhere
\end{itemize}
\end{proof}
\begin{remark}
(Cf.~\cite[p.~392]{LubotzkySegal}.)
Even if $\ad(G)$ is absolutely irreducible, $G$ could 
still be in a maximal subgroup in $\calc_6$.
For example, $\SL(8,5)$ contains the maximal subgroup 
$4\circ 2^{1+6}.\Sp_6(2)\in \calc_6$ which acts 
absolutely irreducibly on $A$; see \cite[p.~399]{Bray}. 
\end{remark}

\begin{theorem}\label{criterion}
There exists a function $f$, depending only on the 
degree $n$, such that $|G|\le f(n)$ for any proper absolutely 
irreducible subgroup $G$ of $\SL(n,p)$ such that $\ad(G)$ is 
irreducible.
\end{theorem}
\begin{proof}
(Cf.~\cite[p.~398]{LubotzkySegal}). 
By~\cite[Section~2.2.6]{Bray}, $L\leq \SL(n,p)$ in $\calc_6$
has order bounded by a function of $n$ only. 
By Proposition~\ref{absirrasch}, then, let $L\in \cals$. 
That is, $L=N_{\SL(n,p)}(K)$ with $K\le\SL(n,p)$ simple non-abelian 
and $C_L(K)=\langle 1_n\rangle$. As $L$ is embedded in $\Aut(K)$, 
a bound on $|K|$ implies a bound on $|L|$.

By the classification of finite simple groups, $K$ can be 
alternating, or of Lie type, or sporadic. Sporadic groups are 
of course bounded in order.

If $K\cong \mathrm{Alt}(k)$ then 
\cite[Theorem~5.7A, corrected]{DixonMortimer} shows that 
$n\ge\frac{2k-6}{3}$; i.e., for fixed $n$, the permutation 
degree $k$ and hence $|K|$ is bounded.

Now let $K=Y_l(r^e)$ for a Lie class $Y$, Lie rank $l$, and $r$ prime.
If $r\not=p$ then \cite[Table 1]{SeitzZalesskii} gives lower 
bounds for the smallest coprime degree $n$ in which $K$ has a 
faithful projective representation. These bounds are functions 
$a(l,r^e)$, independent of $p$, such that $a(l,r^e)\to\infty$ 
as $\l\to\infty$ or $r^e\to\infty$. Thus, in bounded degree $n$,
only a finite number (up to isomorphism) of groups $Y_l(r^e)$ 
are candidates for $K$.

If $r=p$ then \cite[p.~398]{LubotzkySegal} shows that 
$K$ and $L$ must be in a proper connected algebraic subgroup, 
and so do not act irreducibly on the adjoint module $A$.
\end{proof}

\begin{corollary}\label{criterion2}
Let $G \leq \SL(n, p)$, and let $f(n)$ be as in 
Theorem{\em ~\ref{criterion}}. 
If $\ad(G)$ is absolutely
irreducible and $|G| > f(n)$ then $G = \SL(n, p)$.
\end{corollary}
\begin{proof}
Working over the algebraic closure of $\F_p$,
suppose that $G$ is block upper triangular with main diagonal 
$(G_1, G_2)$ where $G_i$ has degree $n_i<n$. 
Then $\ad(G)$ leaves invariant the subspace of the 
adjoint module consisting of all block 
upper triangular matrices with main diagonal 
$(x, 0_{n_2})$, where 
$\trace (x)=0$. Hence $G$ must be absolutely irreducible.
By Theorem~\ref{criterion}, $G = \SL(n, p)$.
\end{proof}

\begin{remark}\label{Interpretfn}
Theorem~\ref{criterion} and Corollary~\ref{criterion2} remain 
valid if we let $f(n)$ be a bound on $\mathrm{exp}(G)$, or a bound 
on the largest order of an element of $G$. 
\end{remark}

Using the formulae for the smallest representation degree of 
alternating groups, and of Lie-type groups in cross-characteristic, 
it would be possible to give a rough upper estimate of $f(n)$. We 
do not attempt this. In Section~\ref{bounds}, we instead
use the tables of \cite[Chapter~8]{Bray} to give tight values for
$f(n)$ in degrees $n\le 12$, extending the values
in~\cite[Remark~3.3]{DensityFurther}. 

\section{Realizing strong approximation computationally}\label{3}

Let $H$ be a dense subgroup of $\SL(n,R)$, $R=\frac{1}{\mu}\Z$.
By Corollary~\ref{criterion2} and Remark~\ref{Interpretfn}, if 
$\ad(\varphi_p(H))$ is absolutely irreducible and $f(n)$ is 
exceeded by $\varphi_p(H)$, then $\varphi_p(H) = \SL(n, p)$.
This result, and a well-known equivalent statement of
density, comprise the background for our main algorithm. 

Input groups for all the algorithms are finitely generated.
Sometimes we write input as a
finite generating set, or as the group itself.

\subsection{Preliminaries}

We start by giving two auxiliary procedures. 
\subsubsection{Bounded order test} 
The first auxiliary procedure is a slight generalization of 
the one in \cite[Section~2.1]{DensityFurther}.
\begin{lemma}\label{exp11}
If $k$ is a positive integer and 
$H \leq \GL(n, R)$ is infinite, then 
$\varphi_p(H)$ has an element of order greater than $k$
for almost all primes $p$.
\end{lemma}
\begin{proof}
The proof is the same as in \cite[Section~2.1]{DensityFurther}.
\end{proof}
\begin{lemma}\label{exp1}
Suppose that $H \leq \SL(n, R)$ and $\varphi_p(H) = \SL(n, p)$
for some prime $p$. If $n \geq 3$
or $p>2$ then $H$ is infinite. 
\end{lemma}
\begin{proof}
See \cite[Lemma~2.1]{DensityFurther}; a finite subgroup of 
$\SL(n, R)$ can be conjugated into $\SL(n, \Z)$.
\end{proof}

The procedure {\tt PrimesForOrder}$(H, k)$ accepts an infinite 
subgroup $H \leq \GL(n, R)$ and a positive integer $k$, and returns 
the finite set of all primes $p$ such that $\varphi_p(H)$ has 
maximal element order at most $k$.
This output obviously contains all primes $p$ such that 
$|\varphi_p(H)|\leq k$.

\subsubsection{Testing absolute irreducibility}

For this subsection, we refer to \cite[p.~401]{Tits} and 
\cite[Section 3.2]{Density}.

Let $N$ be the normal closure $\langle X \rangle^H$
where $X$ is a finite subset of a finitely 
generated group $H \leq \GL(n, \F)$. 
The procedure {\tt BasisAlgebraClosure}$(X, S)$ 
computes a basis $\{A_1, \ldots, A_m\}$ of 
$\langle N \rangle_{\F}$, thereby deciding whether $N$ is 
absolutely irreducible, i.e., whether $m = n^2$. 

The procedure {\tt PrimesForAbsIrreducible} from 
\cite[Section~2.2]{DensityFurther} will operate in the same 
way for absolutely irreducible $H \leq \GL(n, R)$: it
accepts a generating set $S$ of $H$, and returns the (finite) 
set of primes $p$ such that $\varphi_p(H)$ is not absolutely 
irreducible. The first step is to compute a basis of 
$\langle H \rangle_{\Q}$. By making a small adjustment,
we get {\tt PrimesForAbsIrreducible}$(X, S)$; for absolutely 
irreducible $N=\langle X\rangle^H$, it returns the primes $p$ 
such that $\varphi_p(N)$ is not absolutely irreducible.

If $\bar{H} = \varphi_p(H)$ is absolutely irreducible 
(e.g., $\bar{H} = \SL(n, p)$) and 
$\{\bar{A}_1, \ldots, \bar{A}_{n^2}\}$ is a basis of 
$\langle \bar{H} \rangle_{\F_p}$, then $H$ is absolutely 
irreducible and $\{A_1, \ldots, A_{n^2}\}$ is a basis of 
$\langle H \rangle_{\Q}$. Thus, we can simplify
{\tt PrimesForAbsIrreducible} by computing a basis of the 
enveloping algebra over a finite field and then lifting it 
to a basis of $\langle H \rangle_{\Q}$ 
(cf.~\cite[Section 2.2]{DensityFurther}). 
 
\subsection{Density and strong approximation}

Now we give elementary proofs of some properties of dense groups, 
including strong approximation  
(cf.~\cite{Lubotzky97}, \cite[Theorem 9, p.~396]{LubotzkySegal}, 
and \cite[Corollary~3.10]{DensityFurther}). 

The following is fundamental.
\begin{proposition}[{\cite[p.~22]{Rivin1}}]\label{Rivin}
A subgroup $H$ of $\SL(n, \mathbb{C})$ is dense if and only if 
$H$ is infinite and $\ad(H)$ is absolutely irreducible.
\end{proposition}

Let $H$ be a finitely generated subgroup of $\SL(n,R)$.
\begin{lemma}\label{CongAdIsAdCong}
$\varphi_p(\ad(H)) = \ad(\varphi_p(H))$ for all primes $p$ 
(coprime to $\mu$).
\end{lemma}

\begin{corollary}\label{7}
If $\ad(H)$ is absolutely irreducible then $\ad(\varphi_p(H))$ is absolutely 
irreducible for almost all primes $p$.
\end{corollary}

\begin{lemma}\label{absIrr}  
If $\varphi_p(H) = \SL(n, p)$ then $\ad(H)$ is absolutely irreducible. 
\end{lemma}
\begin{proof}
By Lemma~\ref{CongAdIsAdCong},  $\varphi_p(\ad(H)) = \ad(\SL(n, p))$. 
Since the latter is absolutely irreducible,
its preimage $\ad(H)$ is too. 
\end{proof}

\begin{proposition}\label{11} 
The following are equivalent.
\begin{itemize}
\item[{\rm (i)}] $H$ is dense.
\item[{\rm (ii)}] $H$ surjects onto $\SL(n, p)$ for almost all primes $p$.
\item[{\rm (iii)}] $H$ surjects onto $\SL(n, p)$ for some prime $p > 2$.
\end{itemize}
\end{proposition}
\begin{proof}
Suppose that (i) holds. Then by Lemma~\ref{exp11}, Proposition~\ref{Rivin}, 
and Corollary~\ref{7}, $\ad(\varphi_p(H))$ is absolutely irreducible and 
$|\varphi_p(H)|>f(n)$ for almost all primes $p$. 
By Corollary~\ref{criterion2}, $\varphi_p(H) = \SL(n, p)$ 
for such $p$.
 
Suppose that (iii) holds. By Lemma~\ref{absIrr},
 $\ad(H)$ is absolutely irreducible, and by Lemma~\ref{exp1}, 
$H$ is infinite. Therefore $H$ is dense by Proposition~\ref{Rivin}. 
\end{proof}

\section{The main algorithms}
\label{4}

In this section we combine results from Sections~\ref{2} 
and \ref{3} to obtain the promised algorithms to compute
$\Pi(H)$ for dense groups $H$.
These consist of the main procedure, a variation 
aimed at improved performance, and an alternative that 
could be preferable in certain degrees.

Our main procedure, based on 
Corollary~\ref{criterion2}, follows.

\vspace{7pt}

\begin{obeylines}
{\tt PrimesNonSurjectiveSL}

\vspace{2pt}

Input: a finite generating set of a dense group  $H \leq \SL(n, R)$.
Output: $\Pi(H)$.

\vspace{2pt}

1. $\mathcal{P} :=$ {\tt PrimesForOrder}$(H,f(n))\, \cup$ {\tt PrimesForAbsIrreducible}$(\ad(H))$.
2. Return $\{p\in \mathcal{P}\mid \varphi_p(H)\not=\SL(n,p)\}$.
\end{obeylines}

\vspace{9pt}

\noindent Step~2 is performed via standard methods for matrix groups over 
finite fields (e.g., as in \cite{gaprecog}).
\begin{proposition}\label{term}
{\tt PrimesNonSurjectiveSL} 
returns $\Pi(H)$ for dense input $H$.
\end{proposition}
\begin{proof}
Proposition~\ref{Rivin} implies that Step~1 terminates. 
Then $\varphi_p(H) = \SL(n, p)$ for any $p \notin \mathcal{P}$
by Corollary~\ref{criterion2} and Lemma~\ref{CongAdIsAdCong}.
\end{proof}

\subsection{Testing irreducibility}\label{testred}

Testing absolute irreducibility of $\ad(H)$ for $H$ of degree $n$ 
entails computation in degree about $n^4$, which is comparatively 
expensive. However, Theorem~\ref{criterion} offers a
way to bypass this test. That is, we adapt Meataxe 
ideas~\cite{HoltRees96,Parker84} to determine all primes modulo 
which the adjoint representation is merely reducible. 
For simplicity, the discussion will be restricted to  $R=\Z$.

Recall the following special case of Norton's criterion for 
the natural module $V$ of a matrix algebra $\calr$.

\begin{quote}
Suppose that $B\in \calr$ has rank $\rk(B)=n-1$. Assume that $v\calr=V$
for some non-zero $v$ in the nullspace of $B$, and $\calr w=V^\perp$
and for some non-zero $w^\top$ in the nullspace of $B^\top$. 
Then $V$ is irreducible. 
\end{quote}
Now let $\calr \subseteq \mathrm{Mat}(n,\Q)$ be a $\Z$-algebra, 
and suppose that the following hold.
\begin{enumerate}
\item
We have found $B\in \calr$ such that $\rk(B)=n-1$.
\item
For a non-zero $v$ in the nullspace of $B$, the 
$\Z$-span $v\calr$ contains $n$ linearly independent
vectors $v_1,\ldots , v_n$.
\item
For a non-zero $w^\top$ in the nullspace of $B^\top$, 
there are $n$ linearly independent vectors 
$w_1,\ldots,\allowbreak w_n\in \calr w$.
\end{enumerate}

Norton's criterion, applied to the above configuration
modulo $p$, shows that $\varphi_p(\calr)$ is irreducible
 unless
\begin{itemize}
\item[] $\rk(\varphi_p(B))<n-1$, or
\item[] $\varphi_p(v_1),\ldots,\varphi_p(v_n)$ are 
linearly dependent, or
\item[] $\varphi_p(w_1),\ldots,\varphi_p(w_n)$ are 
linearly dependent.
\end{itemize}

To find (a finite superset of) the set of primes $p$ for
which $\varphi_p(\calr)$ is reducible, we form the union of 
three sets, namely the prime divisors of $\det(M_1)$, $\det(M_2)$, 
and $\det(M_3)$, where
\begin{itemize}
\item[] $M_1$ is a full rank
$(n-1)\times(n-1)$ minor of $B$ (modulo other primes, $B$ has rank
$n-1$),
\item[] $M_2$ is the matrix with rows
$v_1,\ldots , v_n$ (modulo other primes, $v$ spans the whole module),
\item[] $M_3$ is the matrix with rows
$w_1,\ldots , w_n$.
\end{itemize}

To make this into a concrete test {\tt PrimesForIrreducible}, let
$\calr=\langle \ad(H)\rangle_\Z$.
Take a small number (say, $100$) of random $\Z$-linear combinations
$B\in \calr$ until a $B$ of rank $n-1$ is detected. Although we do 
not have a justification that such elements occur with sufficient 
frequency, they seem to (as observed in~\cite{ParkerAtlas10YearsOn});
in every experiment so far we found such a $B$.
(Also note that there are irreducible $H$ such that 
$\langle H \rangle_{\Q}$ does not have an element of rank $n-1$;
but if $H$ is absolutely irreducible then such elements always 
exist.)

We now state a version of {\tt PrimesNonSurjectiveSL} 
that may have improved performance in many situations
(see Section~\ref{5}).

\vspace{7pt}

\begin{obeylines}
{\tt PrimesNonSurjectiveSL}, modified.

\vspace{2pt}

1. If {\tt PrimesForIrreducible} confirms that $\ad(H)$ is irreducible then

\hspace{10pt} $\mathcal{P} :=$ {\tt PrimesForOrder}$(H,f(n))\, \cup$ {\tt PrimesForAbsIrreducible}$(H)$ 
\hspace{33pt} $\cup$ {\tt PrimesForIrreducible}$(\ad(H))$;

else 

\hspace{10pt} $\mathcal{P} :=$ {\tt PrimesForOrder}$(H,f(n))\, \cup$ {\tt PrimesForAbsIrreducible}$(\ad(H))$.

2. Return $\{p\in \mathcal{P}\mid \varphi_p(H)\not=\SL(n,p)\}$.
\end{obeylines}

\vspace{7pt}

\begin{proposition}
The above modification of {\tt PrimesNonSurjectiveSL}  
terminates, returning $\Pi(H)$ for input dense $H$.
\end{proposition}
\begin{proof}
This follows from Theorem~\ref{criterion} and Proposition~\ref{term}.
\end{proof}

\begin{remark}
Suppose that {\tt PrimesForIrreducible} completes, i.e., $\ad(H)$ is 
confirmed to be irreducible. Then $H$ is dense if it is infinite and 
absolutely irreducible. This gives a more efficient density test 
than the procedure {\tt IsDenseIR2} in \cite{OWPreprint17}. 
\end{remark}

\subsection{Individual Aschbacher classes}\label{better}

Some Aschbacher classes may not occur in a given degree.
For example, the tensor product classes $\calc_4$ and 
$\calc_7$ are empty in degree $4$. 
Consonant with the approach of 
\cite{DensityFurther}, we show how to determine the primes $p$ such 
that $\varphi_p(H)$ lies in a group in 
$\calc_i\not \in \{\calc_4, \calc_7, \cals\}$,
using tests that do not involve $\ad(H)$.  
The following is vital. 
\begin{lemma}\label{NormalInDenseIsDense}
Let $H \leq \SL(n, \Q)$ be dense. 
If $N \unlhd H$ is non-scalar then $N$ is dense, thus
absolutely irreducible.
\end{lemma}
\begin{proof}
This follows from Proposition~\ref{11}: 
since $N$ is non-scalar, $\varphi_p(N)$ is a normal non-scalar 
subgroup of $\SL(n,p)$ for almost all primes $p$.
\end{proof}

\subsubsection{Testing imprimitivity}\label{C2} 

Suppose that $H\leq \GL(n,\F)$ is imprimitive, so 
$H \leq \GL(a,\F)\wr \mathrm{Sym}(b)$ 
for some $a$, $b>1$ such that $n=ab$. 
If $\mathrm{Sym}(b)$ has exponent $k$ then 
$\langle h^k : h\in H\rangle\leq \GL(a,\F)^b$
is reducible.  
Hence we have the following procedure.

\vspace{7pt}

\begin{obeylines}
{\tt PrimesForPrimitive}

\vspace{2pt}

Input: dense $H = \langle S \rangle\leq \SL(n, \Q)$.
Output: the set of primes $p$ for which $\varphi_p(H)$ is imprimitive.

\vspace{2pt}

1.  Select $h \in H$ such that $h^e$ is non-scalar, where $e =\exp(\mathrm{Sym}(n))$.
2. $\mathcal{P} := {\tt PrimesForAbsIrreducible}(h^e, S)$.
3. Return all $p \in \mathcal{P}$ such that $\varphi_p(H)$ is imprimitive.
\end{obeylines}

\vspace{9pt}

Once more \cite{gaprecog} is used in implementing the last step.
Lemma~\ref{NormalInDenseIsDense}
guarantees termination and correctness of the  output. 

If we happen to know a prime $p$ such that 
$\varphi_p(H) = \SL(n, p)$, then {\tt PrimesForPrimitive} 
simplifies in the familiar way (i.e., by computing 
in a congruence image and then lifting). 

\vspace{7pt}

\begin{obeylines}
{\tt PrimesForPrimitive}, modified.

\vspace{2pt}

1. Let $p$ be a prime for which $\varphi_p(H)=\SL(n,p)$.
2. Find $n^2$ elements $h_i\in H$ such that the $\varphi_p(h_i^k)$ span $\mathrm{Mat}(n,\F_p)$, where $k :=\exp(\mathrm{Sym}(n))$.
3. Return all $p \in$ {\tt PrimesForAbsIrreducible}$(h_1^k, \ldots , h_{n^2}^k)$ such that $\varphi_p(H)$ is 
\ \ \ \  imprimitive.
\end{obeylines}

\vspace{9pt}

The $h_i$ exist by Step~1 and Lemma~\ref{NormalInDenseIsDense}.

\subsubsection{Testing for field extensions}\label{C3} 
The second derived subgroup $G^{(2)}$ of $G \in \calc_3$
is quasisimple and reducible
(\cite[p.~66]{Bray} and \cite[\S4.3]{KleidmanLiebeck}).
Accordingly, {\tt PrimesForReducibleSecondDerived} 
selects a non-scalar double commutator $g$ in the dense 
group $H$ then returns  {\tt PrimesForAbs}-{\tt Irreducible}$(g, S)$.
By Lemma~\ref{NormalInDenseIsDense}, this will yield 
all primes modulo which $H$ is in a group in $\calc_3$.

If we know a prime $p$ such that $\varphi_p(H) = \SL(n, p)$
then {\tt PrimesForReducibleSecond}-{\tt Derived} can be modified 
along the lines of our modification of {\tt PrimesForPrimitive}. 
We search for double commutators (rather than $k$th powers)
in $\varphi_p(H)$ that span $\mathrm{Mat}(n,\F_p)$;
these exist because $\varphi_p(H)=\SL(n,p)$ 
is perfect (if $n>2$ or $p>3$).

\subsubsection{Excluding classes}
For prime $n$ or $n=4$,   
the results of Sections~\ref{C2} and \ref{C3}, 
together with those of \cite{DensityFurther}, enable
us to avoid $\ad(H)$ in computing $\Pi(H)$.
We use the procedures below
to rule out individual Aschbacher classes
in those degrees.

\vspace{2pt}

\begin{description}
\item[$\calc_1$] {\tt PrimesForAbsIrreducible}. 
\item[$\calc_2$] {\tt PrimesForPrimitive}.
\item[$\calc_3$] {\tt PrimesForReducibleSecondDerived}.
\item[$\calc_6$, $\cals$] {\tt PrimesForOrder}.
\item[$\calc_8$] {\tt PrimesForSimilarity}, as in
\cite[Section 2.5]{DensityFurther}.
\end{description}

\section{Experiments}\label{5}

Our algorithms have been implemented in {\sf GAP}~\cite{Gap},
enhancing previous functionality for computing with dense 
groups~\cite{OWPreprint17}. 
The software can be accessed at 

\url{http://www.math.colostate.edu/~hulpke/arithmetic.g} 

\vspace{2pt}

We report on experiments undertaken 
with the implementation.
One major task is computing all congruence quotients of a
finitely generated dense group $H\leq \SL(n,\Z)$ from $\Pi(H)$,
as explained in \cite[Section~4.1]{DensityFurther}. 

\subsection{Explicit order bounds}\label{bounds}

We will let $f(n)$ be a bound on the largest element 
order for the absolutely irreducible groups of degree $n$ in 
$\calc_6\cup \cals$ that are irreducible in their adjoint 
representation. The tables in \cite[Section~8]{Bray} 
furnish bounds for $n\le 12$. We construct an example of each 
such group in $\calc_6\cup\cals$ using the {\sc Magma}~\cite{Magma} 
implementation that accompanies \cite{Bray}. Then we use 
{\sf GAP} to calculate conjugacy class representatives 
and their orders.

For completeness, Table~\ref{tab1} gives maximal subgroup order, 
maximal element order, and the 
least common multiple of exponents.
The column `Geometric' lists the number $i$ of each
Aschbacher class $\calc_i$ that can occur.

We include, for degrees $n\in \{ 3,4,5,7,11\}$, the element order 
bounds from~\cite{DensityFurther} for {\em all} groups in $\calc_6\cup\cals$. 
The rows with these bounds have $n\cals$ in the Degree column. 
For $n = 3,4,5$ the bounds agree, and so we have omitted the row 
beginning with $n$.

\begin{table}[htb]
\begin{center}
\begin{tabular}{l|c|r|r|r}
Degree
&Geometric
&Group Order
&Element order
&Exponent lcm
\\
\hline
$3\cals$&$1,2,3,6,8$&$1080$&$21$&$1260$\\
$4\cals$&$1,2,3,6,8$&$103680$&$36$&$2520$\\
$5\cals$&$1,2,3,6,8$&$129600$&$60$&$3960$\\
$6$&$1,2,3,4,8$&$39191040$&$60$&$2520$\\
$7$&$1,2,3,6,8$&$115248$&$56$&$168$\\
$7\cals$&&$115248$&$84$&$168$\\
$8$&$1,2,3,4,6,8$&$743178240$&$120$&$5040$\\
$9$&$1,2,3,6,7,8$&$37791360$&$90$&$360$\\
$10$&$1,2,3,4,8$&$4435200$&$120$&$9240$\\
$11$&$1,2,3,6,8$&$244823040$&$198$&$637560$\\
$11\cals$&&$244823040$&$253$&$637560$\\
$12$&$1,2,3,4,8$&$5380145971200$&$156$&$360360$\\
\end{tabular}
\end{center}

\medskip

\caption{Order bounds in small degrees}
\label{tab1}
\end{table}

\subsection{Implementation and experimental results}

\subsubsection{Triangle groups} 
Let $\Delta(3,3,4)$ be the triangle group 
$\langle a,b\mid a^3=b^3=(ab)^4=1\rangle$.
In \cite[Theorem~1.1]{LongLatest}, a four-dimensional 
real representation of $\Delta(3,3,4)$ is defined by
\[
\rho_k(a) = \left(
\renewcommand{\arraycolsep}{.2cm}
\begin{array}{cccc}
k(3-4k+4k^2) & -1-4k-8k^2+16k^3-16k^4 & 0 & 0\\
\vspace{-8pt}  \\
1-k+k^2 & -1-3k+4k^2-4k^3 & 0 & 0\\
\vspace{-8pt}  \\
k(1-2k+2k^2) & -3-4k-2k^2+8k^3-8k^4 & 1 & 0\\
\vspace{-8pt}  \\
2(1-k+k^2) & -2(1+2k-4k^2+4k^3) & 0 & 1\end{array}
\right),
\]
\[
\rho_k(b) = \left(\begin{array}{ccrr}
1 & \, 0&-4 & 0\\
0 & \, 1 & 0 &-1\\
0 & \, 0&-1 &-1\\
0 & \, 0 & 1 & 0\end{array}\right).
\] 
Let $H(k)= \langle \rho_k(a), \rho_k(b)\rangle$. 
If $k\in \Z$ then $H(k)\leq \SL(4,\Z)$.

Let $F(k)$ be the image under $\rho_k$ of 
$\langle [a,b], [a, b^{-1}] \rangle$. 
Calculations by D.~F.~Holt (personal communication) using 
{\sf kbmag}~\cite{kbmag}
establishes that the latter is a free subgroup of $\Delta(3,3,4)$. 
All  groups $H(k)$ (resp. $F(k)$) are $2$-generated, and 
of the same structure; as $k$ varies  
we are just changing the size of matrix entries.
Note that the entries of the generators of
$F(k)$ have roughly twice the number of digits as those of $H(k)$.
Our experiments justify that $H(k)$, $F(k)$ are dense 
(for $H(k)$ this follows independently from 
\cite{LongLatest}), and non-arithmetic, i.e., thin. As 
$\calc_4$ and $\calc_7$ do not figure in degree $4$, the 
algorithm from Section~\ref{better} can be utilized here. 
This will illustrate the benefit of the improvements in 
Sections~\ref{testred} and \ref{better}.

In Table~\ref{TableY}, $M$ is the level of $\mathrm{cl}(H)$ and
`Index' is $|\SL(4, \Z) : \mathrm{cl}(H)|$.
We remark that computing $\Pi(H)$, $M$, and indices  
is not possible with our previous 
methods~\cite{DensityFurther,Density}.
Other columns give runtimes in seconds on a 3.7GHz Xeon E5 (2013 MacPro).
Column $t_A$ gives the runtime of {\tt PrimesNonSurjectiveSL}.
Column $t_I$ gives the time of the Meataxe-based algorithm 
from Section~\ref{testred}. 
Due to the randomized nature of the Meataxe calculations, 
timings turned out to be 
variable. Consequently we give a timing
of ten experiments and list minimum, maximum,
and average runtime in the format min--max; average.
Column $t_B$ gives runtimes of the algorithm in
Section~\ref{better} (computing $\Pi(H)$ without $\ad(H)$), 
and the final column $t_M$ is runtime to compute $M$
and Index from $\Pi(H)$.

\begin{table}[h]
\begin{tabular}{l|r|r|r|r|r|r}
$H$
&$M$
&Index
&$t_A$
&$t_I$
&$t_B$
&$t_M$
\\
\hline
$H(1)$
&$2^{5}7^{2}$
&$2^{41}3^{3}5^{3}7^{6}19$
&$63$
&$7{-}69; 27$
&$4$
&$7$
\\
$H(2)$
&$2^{3}313$
&$2^{17}3^{2}5^{2}13
{\cdot}97
{\cdot}101
{\cdot}181^{2}$
&$54$
&$10{-}104; 30$
&$7$
&$1373$
\\
$H(3)$
&$2^{5}7
{\cdot}199$
&$2^{43}3^{6}5^{3}7
{\cdot}11
{\cdot}19
{\cdot}13267
{\cdot}19801$
&$62$
&$9{-}90; 43$
&$7$
&$334$
\\
$H(4)$
&$2^{3}7
{\cdot}607$
&$2^{21}3^{5}5^{5}7
{\cdot}13
{\cdot}19
{\cdot}101
{\cdot}7369
{\cdot}9463$
&$90$
&$22{-}65; 37$
&$19$
&$5938$
\\
$H(5)$
&$2^{5}5^{2}409$
&$2^{44}3^{3}5^{6}17
{\cdot}31
{\cdot}55897
{\cdot}83641$
&$73$
&$13{-}107; 48$
&$11$
&$2883$
\\
$H(6)$
&$2^{3}7
{\cdot}31
{\cdot}97$
&$2^{27}3^{7}5^{5}7
{\cdot}13
{\cdot}19
{\cdot}37$\quad\ &$85$
&$14{-}144; 63$
&$7$
&$308$
\\
&&
\qquad
${\cdot}331{\cdot}941{\cdot}3169$
&&&&\\
$H(10)$
&$2^{3}5^{2}7
{\cdot}919$
&$2^{26}3^{8}5^{8}7^{2}13
{\cdot}17
{\cdot}19^{2}31$\quad\ &$93$
&$67{-}390; 235$
&$14$
&$30382$
\\
&&
\qquad
${\cdot}37
{\cdot}101
{\cdot}113
{\cdot}163$
&&&&\\
$F(1)$
&$2^{5}3^{2}7^{2}$
&$2^{53}3^{8}5^{4}7^{6}19$
&$77$
&$595{-}707; 645$
&$3$
&$16$
\\
$F(2)$
&$2^{4}3^{2}7
{\cdot}13
{\cdot}313$
&$2^{38}3^{9}5^{6}7
{\cdot}13
{\cdot}17
{\cdot}97
{\cdot}101
{\cdot}181^{2}$
&$78$
&$689{-}831; 750$
&$11$
&$5986$
\\
$F(3)$
&$2^{5}3^{2}7
{\cdot}29$
&$2^{62}3^{15}5^{6}7^{3}11
{\cdot}19
{\cdot}67
{\cdot}137$\quad\ &$106$
&$718{-}851; 769$
&$10$
&$10094$
\\
&
\qquad${\cdot}37
{\cdot}199$
&
\qquad
${\cdot}421
{\cdot}13267
{\cdot}19801$
&&&&\\
$F(4)$
&$2^{4}3^{3}7
{\cdot}59
{\cdot}607$
&$2^{37}3^{15}5^{7}7
{\cdot}13
{\cdot}19
{\cdot}29
{\cdot}101$\quad\ &$102$
&$719{-}899; 798$
&$19$
&$74079$
\\
&&
\qquad
${\cdot}1741
{\cdot}7369
{\cdot}9463$
&&&&\\
$F(5)$
&$2^{5}3^{3}5^{2}7$
&$2^{66}3^{15}5^{10}7
{\cdot}17
{\cdot}31
{\cdot}2521$\quad\ &$139$
&$700{-}1010; 881$
&$27$
&$129470$
\\
&\qquad${\cdot}71{\cdot}409$
&
\qquad
${\cdot}55897
{\cdot}83641$
&&&&\\
\end{tabular}

\bigskip

\caption{Experimental data for the groups $H(k), F(k)\leq \SL(4,\Z)$}
\label{TableY}
\end{table}

After computing $M$, we can find all congruence quotients 
of $H(k)$, and hence a set of finite quotients of $\Delta(3,3,4)$.
We see from the results for $k=1$, $2$ that $\Delta(3,3,4)$ 
has quotients $\mbox{PSL}(4, p)$ for $p>2$. On the other hand, 
a calculation with the {\sf GAP} operation {\tt GQuotients} 
shows that $\Delta(3,3,4)$ has no quotient isomorphic to 
$\mbox{PSL}(4,2)$.
Furthermore, since  $\Delta(3,3,4)$ has quotients isomorphic to 
$\mathrm{Alt}(10)$, which cannot be a section of a matrix 
group of degree $4$ over a finite field, 
$H(k)$ is thin for all $k \in \Z$.
The $F(k)$ are thin because they are free.

\subsubsection{Other experiments}\label{exp} 
We used the following constructions of dense groups, including 
examples that permit tensor decomposition modulo some primes.
\begin{itemize}
\item[(i)]  
Let $K(a,b,m)$ be the subgroup of $\SL(ab,\Z)$ generated by 
$\SL(a,\Z)\otimes \SL(b,\Z)$
and the elementary matrix $m t_{1,a+1}$ (two generators 
per factor of the Kronecker product).
\item[(ii)] For distinct monic polynomials 
$p(x),q(x)\in\Z[x]$ of equal degree $n$, let 
$C(p,q)$ be the subgroup of $\SL(n,\Z)$ generated by the 
companion matrices $C_p$ and $C_q$ for $p(x)$ and 
$q(x)$. 
\end{itemize}
Regarding density of the $K(a,b,m)$, cf.~\cite[Lemma~3.15]{DensityFurther}.
By \cite[Theorem~1.5]{Rivin1}, $C(p,q)$ is dense if it is non-abelian, 
$C_q$ has infinite order, and $p(x)$ is irreducible with Galois 
group $\mathrm{Sym}(n)$.

The runtimes in Table~\ref{TableZ} 
have the same interpretation as in Table~\ref{TableY}. 
Some computations with the larger groups did not complete 
for several hours. 
In that event, the pertinent column entry is blank. 
Indices are not listed for space reasons.

\begin{table}[htb]
\begin{tabular}{l|r|r|r|r|r|r}
Group
&Degree
&Primes
&$M$
&$t_A$
&$t_I$
&$t_M$
\\
\hline
$K(2,2,275)$
&$4$
&$5,11$
&$5^{2}11$
&$101$
&$1{-}3;1$
&$8$
\\
$K(2,3,441)$
&$6$
&$3,7$
&$3^{3}7^{2}$
&$37951$
&$4{-}47;17$
&$107$
\\
$K(3,2,8959)$
&$6$
&$17,31$
&$17^{2}31$
&$39873$
&$8{-}43;28$
&$3946$
\\
$K(2,4,100)$
&$8$
&$2,5$
&$2^{4}5^2$%
&&$17{-}96;53$
&$956$\\
$K(3,3,11979)$
&$9$
&$3,11$
&$3^{3}11^3$%
&&$81{-}246;180$
&$4283$\\
$C(x^4{-}x+1, x^4+5x^3{-}x^2+1)$
&$4$
&$11,61$
&$11
{\cdot}61$
&$58$
&$3{-}26;8$
&$2131$
\\
$C(x^6+2x^4+x+1,x^6-x^2+1)$
&$6$
&$7,23$& 
&&$12{-}305;73$
&\\
$C(x^8+x+1,x^8-x+1)$
&$8$
&$2$
&$2^{2}$
&&
$52{-}368;150$&$10$
\\
$C(x^8+2x+1,x^8+x^4+1)$
&$8$
&$2,3,5$
&$2^{4}3{\cdot}5$&&
$33{-}1982;505$&$35813$
\\
\end{tabular}

\bigskip

\caption{Experimental data for the groups $K(a,b,m)$ and
$C(p,q)$}
\label{TableZ}
\end{table}

\subsubsection{Performance}
The runtime to find $\Pi(H)$ is roughly
proportional to the magnitudes of its elements.
In fact, runtime is dominated by tests to ensure 
that no prime $p$ returned is a false positive, 
i.e., that the $p$-congruence image 
really is a proper subgroup of $\SL(n,p)$.

The timings show that the method of Section~\ref{better} 
is clearly superior to the default, with the Meataxe-based 
algorithm performing better unless matrix entries become 
very large. This pattern becomes more pronounced in larger 
degrees.

\subsection*{Acknowledgments}
We thank Mathematisches Forschungsinstitut Oberwolfach 
for facilitation of our work through 
the programme `Research in Pairs' in 2018.
A.~S.~Detinko was supported by Marie Sk\l odowska-Curie Individual 
Fellowship grant H2020 MSCA-IF-2015, no.~704910 (EU Framework 
Programme for Research and Innovation). 
A.~Hulpke was supported by Simons Foundation Collaboration Grant no.~524518 
and National Science Foundation grant DMS-1720146.

\bibliographystyle{amsplain}

\end{document}